\documentclass{amsart}
\usepackage{amssymb}
\usepackage{amsmath}
\usepackage{amscd}
\usepackage[matrix,arrow,curve]{xy}

\newtheorem{theorem}{Theorem}[section]
\newtheorem{lemma}[theorem]{Lemma}

\theoremstyle{definition}
\newtheorem{definition}[theorem]{Definition}

\theoremstyle{remark}

\numberwithin{equation}{section}

\begin{document}

\title{ Inductive Limits for  Systems of Toeplitz Algebras}
\author{R.N.~Gumerov }
\address{Chair of Mathematical Analysis, N.I. Lobachevskii Institute of Mathematics and Mechanics,
 Kazan (Volga region) Federal University, Kremlevskaya~35, Kazan,420008, Tatarstan,
 Russian Federation}
\email{Renat.Gumerov@kpfu.ru; rn.gumeroff@gmail.com}
\thanks{The research was funded by the subsidy allocated to Kazan Federal University
for the state assignment in the sphere of scientific activities, project ¹ 1.13556.2019/13.1.}






\begin{abstract}
 This article
deals with  inductive systems of  Toeplitz algebras over arbitrary
directed sets.  For such a system the family of its connecting
injective $*$-homomorphisms is defined by  a set of natural numbers
satisfying a factorization property. The motivation for the study of
those inductive systems comes from our previous work on the
inductive sequences of Toeplitz algebras defined by sequences of
numbers and the limit automorphisms for the inductive limits of such
sequences. We show that there exists an isomorphism in the category
of unital  $C^*$-algebras and unital $*$-homomorphisms between the
inductive limit of an inductive system of Toeplitz algebras over a
directed set defined by a set of natural numbers and a reduced
semigroup  $C^*$-algebra for a semigroup in the group of all
rational numbers. The inductive systems of Toeplitz algebras over
arbitrary partially ordered  sets defined by  sets of natural
numbers are also studied.
\end{abstract}

\maketitle

\section{Introduction}
The main part of motivation for the present article comes from our
work on inductive sequences of Toeplitz algebras and limit
automorphisms of  $C^*$-algebras generated by isometric
representations for semigroups of rational numbers  (see
\cite{gumerovSMJ2018}). We transfer some results from
\cite{gumerovSMJ2018} concerning inductive sequences of Toeplitz
algebras defined by sequences of numbers to the case of inductive
systems over arbitrary directed sets defined by sets of natural
numbers satisfying a factorization property. In turn, the results in
\cite{gumerovSMJ2018}  are closely related to those in
\cite{ggumerovLJM2002, gumerovPAMS2005, gumerovLJM2005,
gumerovSMJ2013, gumerovRusM2014, UchZBook22018}  which are devoted
to mappings of compact topological groups.

 A part of motivation for studying inductive systems of  $C^*$-algebras
 comes from algebraic quantum field theory  \cite{HaagKastler1964, Haag1996, Araki2009, Horuzhy}.
 The general framework of algebraic quantum field theory is given by a covariant functor. Usually that functor acts from a category associated to a partially ordered set into a category describing the
algebraic structure of observables. The standard assumption in
quantum physics is that the second category consists of unital
$C^*$-algebras and unital   $*$-homomorphisms between  $C^*$-algebras.
Thus one has an inductive system of  $C^*$-algebras.

 A simple example of an inductive system   $\mathcal{F}=(K,\{\mathfrak{A}_a\},\{\sigma_{ba}\})$
  of   $C^*$-algebras over a directed set   $(\,K,\, \leq \,)$  is that in which  $\{\mathfrak{A}_a\mid a\in K\}$
   is \emph{a net of  $C^*$-subalgebras} of a given   $C^*$-algebra  $\mathfrak{A}$.
   By this, one means that each  $\mathfrak{A}_a$  is a  $C^*$-subalgebra containing the unit  $\mathbb{I}_\mathfrak{A}$
    of the algebra  $\mathfrak{A}$,  $\mathfrak{A}_a\subset\mathfrak{A}_b$  and
     $\sigma_{ba}:\mathfrak{A}_a\longrightarrow\mathfrak{A}_b$  is the inclusion mapping whenever  $a,b\in K$
      and  $a\leq b$. Given such a net  $\mathcal{F}$, the norm closure of the union of all  $\mathfrak{A}_a$
       is itself a  $C^*$-subalgebra of  $\mathfrak{A}$  which is a simple example of an inductive
       limit in the category of  $C^*$-algebras.

  The basic tool of the algebraic approach to quantum fields over a spacetime is  a net of
   $C^*$-algebras over a  set defined as a
suitable set of regions of the spacetime ordered under inclusion
\cite{HaagKastler1964, Haag1996, Araki2009}. In \cite{Ruzzi2005,
RuzziVasselli2012, Vasselli2015}  nets consisting of $C^*$-algebras
of quantum  observables for the case of curved spacetimes are
studied. The paper \cite{GLS2016} deals with a net that is
constructed by means of the semigroup  $C^*$-algebra generated by
the path semigroup for  a partially ordered.  In \cite{GLS2018} the
authors consider a net of  $C^*$-algebras associated to a net over a
partially ordered set consisting of Hilbert spaces.

In this article we study inductive limits for systems of Toeplitz
algebras over arbitrary directed sets. Here, by the Toeplitz algebra
we mean the reduced semigroup $C^*$-algebra for the additive
semigroup of non-negative integers. We recall that the reduced
semigroup  $C^*$-algebra  can be constructed for an arbitrary left
cancellative semigroup. This algebra is a very natural object
because it is generated by the left regular representation of the
left cancellative semigroup.The study of such semigroup
$C^*$-algebras goes back to L.~A.~Coburn \cite{coburn67, coburn69},
R.~G.~Douglas \cite{douglas72}, G.~J.~Murphy \cite{murphy87,
murphy91}  and is continued at the present time (see, for example,
\cite{LipSibMJ, Liarxiv} and references there in).

To define a family of  connecting injective  $*$-homomorphisms for
an inductive system of Toeplitz algebras over a directed set we make
use of a set of  natural numbers satisfying factorization equalities
(see Section~\ref{sectionILSTA},
Definition~\ref{D:indsystemToeplitzalg})   and Coburn's Theorem
\cite[Theorem~3.5.18]{murphy}.

The main result, Theorem~\ref{indlim},  states  that the inductive
limit  for  an inductive system of Toeplitz algebras over a directed
set defined by a set of natural numbers is isomorphic in the
category of unital  $C^*$-algebras and their unital
$*$-homomorphisms to a reduced semigroup  $C^*$-algebra for a
semigroup in the group of all rational numbers.

The article is organized as follows. It consists of  Introduction,
Preliminaries  and three more sections containing the results.
Section~\ref{sectionILSTA} deals with the  auxiliary statements that
are used  for proving the main result. In Section~\ref{main} is
devoted to the proof of Theorem~\ref{indlim}. Section~\ref{last}
contains the theorem on inductive systems of Toeplitz algebras over
arbitrary partially ordered sets defined by  sets of natural numbers
satisfying factorization equalities.

The results in this article were announced without proofs in
\cite{GLG2018} and are closely related to those in \cite{GLG2019}.

\section{Preliminaries}\label{preliminaries}

As usual,  $\mathbb{N}$, $\mathbb{Z}$, $\mathbb{Q}$  and
$\mathbb{C}$  denote the set of all natural numbers, the additive
group of all integers, the additive group of all rational numbers
and the field of all complex numbers respectively. For a sequence of
numbers \ $M=(m_1, m_2,  \ldots)$, where  $m_s\in \mathbb{N}$,
$s\in\mathbb{N}$,   we shall consider the subgroup $\mathbb{Q}_M$ of
the group  $\mathbb{Q}$ defined as follows:
$$
\mathbb{Q}_M:=\left\{\frac{m}{m_1m_2\ldots m_s}\mid m\in \mathbb{Z}, s\in\mathbb{N} \right\}.
$$
It is known  (see, for example, \cite[Proposition~1]{murphy91},
\cite[Lemma~1]{bogatyifrolkina})  that the group   $\mathbb{Q}_M$ is
the inductive (direct) limit in the category of groups and their
homomorphisms for the following inductive (direct) sequence
\begin{equation}\label{directseqZ}
\CD
\mathbb{Z}  @>\tau_1>>\mathbb{Z}  @>\tau_2>> \mathbb{Z}  @>\tau_3>> \ldots,
\endCD
\end{equation}
where the connecting homomorphisms  $\tau_s$  are given by
$\tau_s(m)=m_sm,$ \ $ m\in \mathbb{Z},$   $s\in\mathbb{N}$. By a
directed set we mean an upward directed partially ordered set.
 We recall the definition of  reduced semigroup
$C^*$-algebras for semigroups in  $\mathbb{Q}$. To do this we assume
that  $\Gamma$  is an arbitrary subgroup in the group  $\mathbb{Q}$.
The positive cone in the ordered group $\Gamma$
 is denoted by the symbol
$$
\Gamma^+:=\Gamma\cap
[0,+\infty)\,.
$$
As usual, the symbol  $l^2(\Gamma^+)$  stands for the Hilbert space
of all square summable complex-valued functions on the additive
semigroup $\Gamma^+$:
$$
   l^2(\Gamma^+):=\{f:\Gamma^+\rightarrow \mathbb{C}:\sum_{\gamma\in \Gamma^+} |f(\gamma)|^2 < +\infty \}.
$$
Recall that   the inner product in the space  $l^2(\Gamma^+)$  is
given  by the formula
   $
   <f,g>:=\sum_{\gamma\in \Gamma^+} f(\gamma)\overline{g(\gamma)}.
   $
The canonical orthonormal basis in the Hilbert space $l^2(\Gamma^+)$
is denoted by  $\{\, e_g \mid g\in \Gamma^+ \, \}$. That is, for
arbitrary elements  $g,h\in \Gamma^+$,
 we set  $e_g(h)=\delta_{g,h}$, where
 \begin{equation*}\label{Def:ea}
\delta_{g,h}:=
\begin{cases}
1, &\text{if $g=h$\,;}\\
0, &\text{if $g\neq h$\,.}
\end{cases}
\end{equation*}

 Let us consider the $C^*$-algebra  of all bounded linear operators $B(l^2(\Gamma^+))$
 in the Hilbert space $l^2(\Gamma^+)$.
For every element $g\in \Gamma^+$, we define the isometry $V_g\in
B(l^2(\Gamma^+))$ by
$
V_ge_h:=e_{g+h},
$
 where  $h$  is an element of the semigroup  $ \Gamma^+.$

  We denote by  $C^*_r(\Gamma^+)$  the  $C^*$-subalgebra in the  $C^*$-algebra  $B(l^2(\Gamma^+))$  generated
 by the set of isometries  $\{\, V_g \, \mid \, g\in \Gamma^+ \, \}$. The algebra   $C^*_r(\Gamma^+)$  is called
  \textit{the reduced semigroup \ $C^*$-algebra of the semigroup $\Gamma^+$}, or
 \textit{the~Toeplitz algebra generated by  $\Gamma^+$}.
As was mentioned in Introduction, in the similar way a semigroup
$C^*$-algebra can be defined for an arbitrary left cancellative
semigroup (see, for example,  \cite[Section~2]{Liarxiv}). In the
case when  $\Gamma$ is the group   $\mathbb{Z}$, we also denote the
semigroup  $C^*$-algebra  $C^*_r(\mathbb{Z}^+)$  by $\mathcal{T}$
and use the symbols  $T$  and \ $T^n$  instead of $V_1$  and  $V_n$,
respectively, where  $n\in \mathbb{Z^+}$.

The following statement is an immediate consequence of Coburn's
theorem \cite[Theorem 3.5.18]{murphy}.

\begin{lemma}\label{CoburnThCor}
For every number \ $n\in \mathbb{N}$, there exists a unique unital
$*$-homo\-morphism of \ $C^*$-algebras
$\varphi:\mathcal{T}\longrightarrow \mathcal{T}$  such that  $
\varphi(T)=T^n$. Moreover, $\varphi$ is isometric.
\end{lemma}

   In the sequel, we abbreviate those self-homomorphisms of the Toeplitz algebra as follows:
\[
\varphi:\mathcal{T}\longrightarrow \mathcal{T}:T\longmapsto T^n.
\]
We note that a straightforward proof of Lemma~\ref{CoburnThCor} is
also contained in \cite[Proposition~3]{gumerovIOP}. For necessary
results in the theory of  $C^*$-algebras we refer the reader, for
example, to books  \cite{murphy}  and  \cite[Ch.\,4,
\S\,7]{helemskii}.

Further, we recall the definition and  some facts concerning the
inductive limits for inductive systems of  $C^*$-algebras (see, for
example,  \cite[Section 11.4]{KadisonRingrose}). Necessary facts
from the theory of categories and functors are contained, for
example, in  \cite[ Ch.\,0, \S\,2]{helemskii}  and
\cite{BucurDeleanu1968}.

In what follows, up to Section~\ref{last}, we shall consider a
directed set  $(\,K, \, \leq \,)$. Note that in Section~\ref{last}
we will deal with  an arbitrary partially ordered set   $(\,K, \,
\leq \,)$    that is not necessarily directed.

 The
category associated to the set  $(\,K, \, \leq \,)$   is denoted by
the same letter  $K$. We recall that the objects of this category
are the elements of the set  $K$, and for any pair  $a,b \in K$ the
set of morphisms  $Mor_K(a,b)$  from  $a$  to   $b$  is defined as
follows:
\begin{equation*}\label{Mor(a,b)}
Mor_K(a,b)=
\begin{cases}
\{(a,b)\}, &\text{if $a\leq b$\,;}\\
\emptyset, &\text{otherwise\,.}
\end{cases}
\end{equation*}

Let us take a covariant functor  $\mathcal{F}$  from the category
$K$  into the category of unital  $C^*$-algebras and their unital
$*$-homomorphisms. Such a functor is called \emph{an inductive
system} in the category of  $C^*$-algebras over the set  $(K, \,
\leq~)$. It may be given by a collection
$(K,\{\mathfrak{A}_a\},\{\sigma_{ba}\})$  satisfying the properties
from the definition of a functor.  We shall write
$\mathcal{F}=(K,\{\mathfrak{A}_a\},\{\sigma_{ba}\})$.
 Here, $\{\mathfrak{A}_a\mid a\in K\}$ \ is a family of unital $C^*$-algebras.
 We also suppose that all morphisms $\sigma_{ba}:\mathfrak{A}_a\longrightarrow \mathfrak{A}_b$,
 where  $a\leq b$, are embeddings of
$C^*$-algebras, i.\,e., unital injective $*$-homomorphisms. Recall
that the diagram
\[
\xymatrix{
  \mathfrak{A}_a \ar[rr]^{\sigma_{ca}} \ar[dr]_{\sigma_{ba}}
                &  &    \mathfrak{A}_c     \\
                & \mathfrak{A}_b \ar[ur]_{\sigma_{cb}}                }
\]
is commutative for  all elements  $a,b,c \in K$  satisfying the
conditions  $a\leq b$  and  $ b\leq c$, that is, the following
equation holds:
\begin{equation}\label{sigmacasigmacbsigmaba}
\sigma_{ca}=\sigma_{cb}\circ\sigma_{ba}.
\end{equation}
Furthermore, for each element  $a\in K$  the morphism
$\sigma_{aa}:\mathfrak{A}_a\longrightarrow \mathfrak{A}_a$ is the
identity mapping. In the case of a countable set  $K$ the system
$\mathcal{F}$  is called \emph{an inductive sequence} of
$C^*$-algebras.

\emph{The inductive limit} of the system
$\mathcal{F}=(K,\{\mathfrak{A}_a\},\{\sigma_{ba}\})$  over a
directed set  $K$  is a pair  $(\mathfrak{A}, \{\sigma^K_a\})$,
where
 $\mathfrak{A}$  is a  $C^*$-algebra and  $\{\sigma^{K}_a:\mathfrak{A}_a\rightarrow \mathfrak{A} \mid a\in
K\}$  is a family of unital injective  $*$-homomorphisms such that
the following two properties are fulfilled  \cite[Proposition
11.4.1]{KadisonRingrose}:

1) For every pair of elements  $a,b\in K$  satisfying the condition
 $a\leq b$  the diagram
\[
\xymatrix{
  \mathfrak{A}_a \ar[rr]^{\sigma_{ba}} \ar[dr]_{\sigma^{K}_a}
                &  &    \mathfrak{A}_b \ar[dl]^{\sigma^{K}_b}    \\
                & \mathfrak{A}                 }
\]
is commutative, that is, the equality for mappings
\begin{equation}\label{sigma}
\sigma^{K}_a=\sigma^{K}_b\circ\sigma_{ba}
\end{equation}
holds. Moreover, we have the~following equality:
\begin{equation}\label{cup}
\mathfrak{A}=\overline{\mathop\bigcup\limits_{a\in
K}\sigma^{K}_a(\mathfrak{A}_a)},
\end{equation}
where the bar means the closure of the set with respect to the norm
topology in the  $C^*$-algebra  $\mathfrak{A}$.

2)  If  $\mathfrak{B}$  is a  $C^*$-algebra,
$\psi_a:\mathfrak{A}_a\longrightarrow\mathfrak{B}$ is an injective
$*$-homomorphism for each  $a\in K$, and conditions analogous to
those in  (\ref{sigma})  and  (\ref{cup})  are satisfied, then there
exists a~unique  $*$-isomorphism $\theta$  from $\mathfrak{A}$ onto
 $\mathfrak{B}$  such that the diagram
\begin{equation*}
\xymatrix{
                & \mathfrak{A}_a\ar[dr]^{\psi_a} \ar[dl]_{\sigma^{K}_a}            \\
 \mathfrak{A}  \ar[rr]^{\theta} & &     \mathfrak{B}        }
\end{equation*}
is commutative for every $a\in K$,  that is, the equality
$\psi_a=\theta\circ\sigma^{K}_a $ holds.

The  $C^*$-algebra  $\mathfrak{A}$  itself is often called the
inductive limit. It is denoted as follows: $ \varinjlim\mathcal{F}:=
\mathfrak{A}. $ As is well known, the inductive limit can always be
constructed for an inductive system in the category of
$C^*$-algebras and their $*$-homomorphisms. For the details we refer
the reader to Proposition 11.4.1 in \cite{KadisonRingrose}.

\section{Auxiliary results}\label{sectionILSTA}

 The inductive sequences of Toeplitz algebras defined by arbitrary sequences
  of prime numbers are the objects for studying in  \cite{gumerovSMJ2018}. Now,
   by analogy with the definition of such a sequence (see \cite[Definition~1]{gumerovSMJ2018}),
    we are going to give the definition of the inductive system of Toeplitz algebras over a partially
    ordered set defined by a set of natural numbers possessing an additional property.

Let  $(\,K, \, \leq \,)$  be a directed set. In what follows, we
consider a set of natural numbers
 \begin{equation}\label{setN}
 N=\{n_{ba}\in \mathbb{N} \mid a, b\in K, a\leq b \}
 \end{equation}
 such that the factorization equalities
\begin{equation}\label{factor}
n_{ca}=n_{cb}\cdot n_{ba}
\end{equation}
hold for all elements  $a, b, c\in K$  satisfying the conditions
$a\leq b$ and $b\leq c$. It follows immediately from (\ref{factor})
 that the equality  $n_{aa}=1$  holds for every $a\in K$.

Further, using Lemma~\ref{CoburnThCor}, for every number  $n_{ba}\in
N$  we define the isometric  $*$-homomorphism by the formula
\begin{equation}\label{sigmaba}
\sigma_{ba}:\mathcal{T}\longrightarrow \mathcal{T}:T\longmapsto
T^{n_{ba}}.
\end{equation}
It is clear that the equalities  (\ref{sigmacasigmacbsigmaba}) are
valid for all elements  $a, b,c\in K$   whenever  the conditions
$a\leq b$  and  $b\leq c$  hold, and for each  $a\in K$  the
$*$-homomorphism $\sigma_{aa}:\mathfrak{A}_a\longrightarrow
\mathfrak{A}_a$  is the identity mapping. Thus we can give the
definition of an inductive system of  Toeplitz algebras (compare
with \cite[Definition~1]{gumerovSMJ2018}) over a directed set
defined by a set of natural numbers satisfying factorization
equalities.

\begin{definition}\label{D:indsystemToeplitzalg}
Let  $(\,K, \, \leq \,)$  be a directed  set and $N$  be a  set of
natural numbers  (\ref{setN})  satisfying   (\ref{factor}). An
inductive system of  $C^*$-algebras
\begin{equation}\label{mathcalFindsystemoverK}
\mathcal{F}=(K,\{\mathcal{T}_a\},\{\sigma_{ba}\}),
\end{equation}
where  $\mathcal{T}_a=\mathcal{T}$  for all  $ a\in K$  and the
connecting  $*$-homomorphisms  $\sigma_{ba}$  are given by
(\ref{sigmaba}), is called the inductive system of Toeplitz algebras
over    $K$   defined by  $N$.
\end{definition}

To obtain the main result of the article  we  shall prove in this
section several auxiliary assertions. For  formulations of these
assertions we introduce  additional notation.

Firstly, for a directed  set   $K$  and its arbitrary element  $a\in
K$   the cofinal subset   $K^a$  of   $K$  is defined as follows:
$
K^a:=\{b\in K \mid a\leq b\}.
$

Secondly, for a subset   $S$  in the set  $K^a$  we shall deal with
 the set of natural numbers
$$
N(S):=\left\{n_{ba}\in N \mid b\in  S \right\}.
$$
 Throughout this section     $(K,\{\mathcal{T}_b\},\{\sigma_{cb}\})$  is an inductive system  of Toeplitz algebras
over  $K$  defined by set  (\ref{setN})  satisfying  (\ref{factor}).
Moreover, for an element  $a\in K$  we consider the inductive system
\begin{equation}\label{systemoverKa}
\left(K^a,\{\mathcal{T}_b\},\{\sigma_{cb}\}\right)
\end{equation}
of Toeplitz algebras over   $K^a$  defined by the set
\begin{equation}\label{NsetforKa}
\left \{n_{cb}\in N \mid  b,c\in K^a\right\}.
\end{equation}
Using  property~2)  of the inductive limit for the system
  (\ref{systemoverKa}) (see Preliminaries),  it is straightforward
to prove the following statement.

\begin{lemma}\label{l1}
For every element  $a\in K$   there exists an isomorphism  between
the inductive limits
$$
\varinjlim(K^a,\{\mathcal{T}_b\},\{\sigma_{cb}\}) \, \simeq \, \varinjlim
(K,\{\mathcal{T}_b\},\{\sigma_{cb}\})
$$
in the category of unital   $C^*$-algebras and unital
$*$-homomorphisms.
\end{lemma}

\begin{lemma}\label{l2}
The following two assertions are valid\/${:}$

{\rm 1)} \, if  for elements   $b,c\in K^a$   the equality for
numbers
\begin{equation}\label{eq:nbanca}
n_{ba}=n_{ca}
\end{equation}
 holds   then
we have the equality for the images of Toeplitz algebras
\begin{equation}\label{eq:sigmabK}
\sigma_b^{K^a}(\mathcal{T}_b)=\sigma_c^{K^a}(\mathcal{T}_c);
\end{equation}

{\rm 2)} \, if for some  $k\in\mathbb{N}$  and   $b,c\in K^a$  the
equality for numbers
\begin{equation}\label{eq:ncaknba}
 n_{ca}=k\cdot n_{ba}
 \end{equation}
   holds    then  we  have  the   inclusion for the images of Toeplitz algebras
 \begin{equation}\label{inclusionsigmabK}
\sigma_b^{K^a}(\mathcal{T}_b)\subset\sigma_c^{K^a}(\mathcal{T}_c).
 \end{equation}
\end{lemma}

\begin{proof} 1)  Assume that for some elements  $b,c \in K^a$
  equality  (\ref{eq:nbanca}) holds. Since  $K$ is a directed set there exists an element
   $d\in K^a$  such that  $b\leq d$  and  $c\leq d$.
Then factorization equalities (\ref{factor}) yield the following
equalities for the corresponding natural numbers:
\begin{equation}\label{ndb}
n_{db}\cdot n_{ba}=n_{da}=n_{dc}\cdot n_{ca}.
\end{equation}
Therefore, using condition   (\ref{eq:nbanca}), one gets immediately
the equality $ n_{db}=n_{dc},$
  which implies the following equality for the images of Toeplitz algebras:
\begin{equation}\label{eq:sigmadb}
\sigma_{db}(\mathcal{T}_b)=\sigma_{dc}(\mathcal{T}_c).
\end{equation}
 By  (\ref{sigma})  and  (\ref{eq:sigmadb}), we obtain  desired relation  (\ref{eq:sigmabK}):
$$
\sigma^{K^a}_b(\mathcal{T}_b)=\sigma^{K^a}_d(\sigma_{db}(\mathcal{T}_b))=
\sigma^{K^a}_d(\sigma_{dc}(\mathcal{T}_c))=\sigma_c^{K^a}(\mathcal{T}_c).
$$

2) \, Again, we choose an element    $d\in K^a$  such that the
conditions   $b\leq d$   and   $c\leq d$  are satisfied.
 Then, by  (\ref{ndb})  and  (\ref{eq:ncaknba}), we get the
equality
 $n_{db}\cdot n_{ba}=n_{dc}\cdot k\cdot n_{ba}.
 $
Consequently, we have the  equality  $n_{db}=n_{dc}\cdot k $
 that
together with  (\ref{sigmaba})  imply the following inclusion for
the images of Toeplitz algebras:
\begin{equation}\label{inclusionsigmadb}
\sigma_{db}(\mathcal{T}_b)\subset\sigma_{dc}(\mathcal{T}_c).
\end{equation}
Thus, using equality  (\ref{sigma})  and inclusion
 (\ref{inclusionsigmadb}), we obtain  required inclusion
 (\ref{inclusionsigmabK}):
$$
\sigma^{K^a}_b(\mathcal{T}_b)=\sigma^{K^a}_d(\sigma_{db}(\mathcal{T}_b))\subset
\sigma^{K^a}_d(\sigma_{dc}(\mathcal{T}_c))=\sigma_c^{K^a}(\mathcal{T}_c).$$
\end{proof}

\begin{lemma}\label{l3} There exists a totally ordered countable subset  $\Lambda^a$  in the set
 $K^a$   satisfying the following property. For every element  $b \in
K^a$  there is an element  $c\in \Lambda^a$  and a number  $k\in
\mathbb{N}$   such that the  equality
\begin{equation}\label{ncakdotnba}
n_{ca}=k\cdot n_{ba}
\end{equation}
holds, where   $n_{ca}\in N(\Lambda^a)$,   $n_{ba}\in N(K^a) $.
\end{lemma}

\begin{proof}
First of all, in the set of natural numbers   $N(K^a)$   we consider
the subset
\begin{equation}\label{set:N0Ka}
\left\{n_{b_sa}\in N \mid s\in \mathbb{N} \right\}
\end{equation}
which is uniquely determined  by the following three properties:
\begin{itemize}
\item $b_1=a$;
\item the inequality  $n_{b_sa}<n_{b_{s+1}a}$  is valid for every   $s\in \mathbb{N}$;
\item for every number  $ n_{ba}\in N(K^a)$  there exists a natural number  $n_{b_sa}$  in set
 (\ref{set:N0Ka})  such that   the~equality
    $
    n_{ba}=n_{b_sa}
    $
 holds.
\end{itemize}

In other words, we throw out repeating numbers from the set
$N(K^a)$. Moreover, the elements of set  (\ref{set:N0Ka}) constitute
an increasing sequence of natural numbers  indexed by  $s$. To
construct the desired set  $\Lambda^a$  we use set (\ref{set:N0Ka})
and proceed as follows.

As the first element of the set  $\Lambda^a$ denoted by  $c_1$  we
choose the element  $a$. We note that the condition  $b_1=a\leq
c_1=a$  is satisfied. Then   equality  (\ref{ncakdotnba})  holds
with  $ c=c_1=a$, $b=b_1=a  $   and  $ k=n_{c_1b_1}=n_{aa}=1. $

As the second element of the set  $\Lambda^a$  we take any element
$c_2\in K^a$  satisfying the conditions  $c_1\leq c_2$ and $b_2\leq
c_2$. Such an element exists because   $K^a$ is a directed set. In
this case we have  equality  (\ref{ncakdotnba})  with  $ c=c_2,$
$b=b_2$  and   $ k=n_{c_2b_2}. $ Continuing to argue in this way, we
shall construct  a countable totally ordered subset
\begin{equation}\label{Lambdaa}
\Lambda^a:=\{c_s \in K^a \mid s\in \mathbb{N} \}
\end{equation}
in the set
 $K^a$  possessing the required property.
It is worth noting that the condition  $c_s\leq c_{s+1}$  holds for
every  $s\in \mathbb{N}$.
\end{proof}

In the next lemma we consider the
inductive sequence of Toeplitz algebras
\begin{equation}\label{systemoverLambda}
\left(\Lambda^a,\{\mathcal{T}_{c_s}\},\{\sigma_{c_tc_s}\}\right)
\end{equation}
 over  the set  $\Lambda^a$  defined by the
subset  $\left \{n_{cb}\in N \mid  b,c\in \Lambda^a\right\}$  in the
set  $N$. The proof of the following statement is similar to that of
Proposition~1 in \cite{gumerovSMJ2018}.

\begin{lemma}\label{l4} There exists a subgroup  $\mathbb{Q}_M$
in the group \ $\mathbb{Q}$  such that the reduced semigroup
$C^*$-algebra of the semigroup  $\mathbb{Q}_M^+$  is isomorphic to
the inductive limit of the inductive sequence
(\ref{systemoverLambda}):
\begin{equation}\label{isoCredindlimsystemoverLambda}
C^*_{r}(\mathbb{Q}_M^+) \, \simeq \, \varinjlim\limits \left(\Lambda^a,\{\mathcal{T}_{c_s}\},\{\sigma_{c_tc_s}\}\right).
\end{equation}
\end{lemma}

\textbf{Remark.} 
In \cite[Theorem~1]{murphy87} it is shown that the functor sending a
partially ordered group to the corresponding Toeplitz algebra is
continuous. Consider the inductive sequence of groups
(\ref{directseqZ}) with the connecting homomorphisms  $\tau_s$ given
by  $\tau_s(m)=n_{c_{s+1}c_s}m,$  $ m\in \mathbb{Z},$
$s\in\mathbb{N}$. Applying the above-mentioned functor to this
inductive sequence of groups and making use of Theorem~1 in
\cite{murphy87}, we obtain another proof of Lemma~\ref{l4}.

\section{The main result}\label{main}

In this section we prove the following statement concerning  the
inductive limits for inductive systems  of Toeplitz algebras  over
directed sets defined by sets of natural numbers (\ref{setN})
satisfying factorization equalities (\ref{factor}). To do this we
shall use the results from the previous section.

\begin{theorem}\label{indlim}
Let   $\mathcal{F}$  be an inductive system of Toeplitz algebras
over a directed set  $K$  defined  by a  set of natural numbers $N$
 satisfying factorization equalities.  Then there exists a subgroup
 $\mathbb{Q}_M$  in the group of all rational numbers such that
the reduced semigroup  $C^*$-algebra of the semigroup
$\mathbb{Q}_M^+$  is isomorphic to the inductive limit of the
inductive system  $\mathcal{F}$:
\begin{equation}\label{mainiso}
C^*_{r}(\mathbb{Q}_M^+) \, \simeq \, \varinjlim\limits \mathcal{F}.
\end{equation}
\end{theorem}

\begin{proof}
  As in the previous section we use the notation   $\mathcal{F}=(K,\{\mathcal{T}_a\},\{\sigma_{ba}\})$.

 Let us fix an element  $a\in K$. Then we take a totally ordered countable subset
 $\Lambda^a$  in the set  $K^a$  which is constructed in the proof of Lemma~\ref{l3} (see (\ref{Lambdaa})).

We claim that there exists an isomorphism between the inductive limits of the inductive system
 $\mathcal{F}$  and  inductive sequence   (\ref{systemoverLambda})  of Toeplitz algebras, that is,
\begin{equation}\label{iso1}
\varinjlim\limits \mathcal{F} \, \simeq \,
 \varinjlim\limits\left(\Lambda^a,\{\mathcal{T}_{c_s}\},\{\sigma_{c_tc_s}\}\right)
\end{equation}
in the category of unital  $C^*$-algebras and  unital
$*$-homomorphisms.

Indeed, we consider the inductive system  (\ref{systemoverKa})  of
Toeplitz algebras over  the set  $K^a$  defined by the subset
(\ref{NsetforKa})  in the set  $N$. By  (\ref{cup})  we have the
following equality for the inductive limit of this system:
\begin{equation}\label{indlimsystemoverKa}
\varinjlim\limits\left(K^a,\{\mathcal{T}_b\},\{\sigma_{cb}\}\right) \, = \,
\overline{\bigcup\limits_{b\in
K^a}\sigma^{K^a}_b(\mathcal{T}_b)}.
\end{equation}

By Lemma~\ref{l1}, to show the existence of isomorphism (\ref{iso1})
it is enough to prove that one has the isomorphism between the
inductive limits of systems  (\ref{systemoverLambda})  and
(\ref{systemoverKa})  in the category of unital  $C^*$-algebras and
their unital  $*$-homomorphisms:
\begin{equation}\label{isoindlimLambdaaKa}
\varinjlim\limits\left(\Lambda^a,\{\mathcal{T}_{c_s}\},\{\sigma_{c_tc_s}\}\right)
\, \simeq \,
\varinjlim\limits\left(K^a,\{\mathcal{T}_b\},\{\sigma_{cb}\}\right).
\end{equation}
To construct an isomorphism  (\ref{isoindlimLambdaaKa}) we shall use
the universal property and property~2) of the inductive limit for
system   (\ref{systemoverLambda}) (see Preliminaries). To this end,
for elements $c_s,c_t\in \Lambda^a$   satisfying the condition
$c_s\leq c_t$  we consider the following commutative diagram
(firstly without the dashed arrow):
\begin{equation}\label{diagram:isoinlimLamdaaKa}
\vcenter{
\xymatrix{
 \mathcal{T}_{c_s}\ar[ddddrrr]_{\sigma^{K^a}_{c_s}}
 \ar[rrrrrr]^{\sigma_{c_tc_s}} \ar[rrrdd]^{\sigma^{\Lambda^a}_{c_s}}
   & & & & & & \mathcal{T}_{c_t}\ar[ddddlll]^{\sigma^{K^a}_{c_t}} \ar[llldd]_{\sigma^{\Lambda^a}_{c_t}} \\
\\
 & & & \varinjlim\limits\left(\Lambda^a,\{\mathcal{T}_{c_s}\},\{\sigma_{c_tc_s}\}\right)
 \ar@{-->}[dd]^{\theta}  &  & & & \\
 \\
& & & \varinjlim\limits\left(K^a,\{\mathcal{T}_b\},\{\sigma_{cb}\}\right) & & & }
}
\end{equation}
The universal property of the inductive limits for systems of
$C^*$-algebras  and the injectivity of  $*$-homomorphisms in diagram
(\ref{diagram:isoinlimLamdaaKa})  yield the injective
$*$-homomorphism
$$
\theta:\varinjlim\limits\left(\Lambda^a,\{\mathcal{T}_{c_s}\},\{\sigma_{c_tc_s}\}\right) \longrightarrow
\varinjlim\limits\left(K^a,\{\mathcal{T}_b\},\{\sigma_{cb}\}\right)
$$
such that diagram (\ref{diagram:isoinlimLamdaaKa})  complemented by
 $\theta$  is  also commutative.
To show that the homomorphism  $\theta$  is surjective it is
sufficient  (see property~2)  in Preliminaries)  to prove the
equality
\begin{equation}\label{indlimsystemoverKathroughLamdaa}
\varinjlim\limits\left(K^a,\{\mathcal{T}_b\},\{\sigma_{cb}\}\right) \, = \,
\overline{\bigcup\limits_{s=1}^{+\infty}\sigma^{K^a}_{c_s}(\mathcal{T}_{c_s})}.
\end{equation}
Since we have the inclusion of sets  $\Lambda^a\subset K^a$ and
representation  (\ref{indlimsystemoverKa}), the space on the
left-hand side of  (\ref{indlimsystemoverKathroughLamdaa}) contains
the space on its right-hand side.

To prove the reverse inclusion for the spaces in
(\ref{indlimsystemoverKathroughLamdaa})  we use
(\ref{indlimsystemoverKa})  again. For this aim let us fix an
arbitrary element  $b\in K^a$. We state that the image set
$\sigma^{K^a}_b(\mathcal{T}_b)$ is contained in the set
$\sigma^{K^a}_{c_s}(\mathcal{T}_{c_s})$  for an element  $c_s\in
\Lambda^a$.

Really, by Lemma~\ref{l3}, there exists an element  $c_s\in
\Lambda^a$  and a number  $k\in \mathbb{N}$  such that the
factorization equality  (\ref{ncakdotnba})  is valid for natural
numbers  $n_{c_sa}\in N(\Lambda^a)$,  $n_{ba}\in N(K^a) $. By
assertion~{\rm 2)} in Lemma~\ref{l2}, for the image sets
$\sigma^{K^a}_b(\mathcal{T}_b)$  and
$\sigma^{K^a}_{c_s}(\mathcal{T}_{c_s})$  we obtain  inclusion
(\ref{inclusionsigmabK}) with  $c_s$  instead of  $c$.
 It follows from  (\ref{indlimsystemoverKa})  that
the inclusion for  $C^*$-algebras
\begin{equation*}\label{incl}
\varinjlim\limits\left(K^a,\{\mathcal{T}_b\},\{\sigma_{cb}\}\right) \, \subset \,
\overline{\bigcup\limits_{s=1}^{+\infty}\sigma^{K^a}_{c_s}(\mathcal{T}_{c_s})}
\end{equation*}
holds. Hence, equality  (\ref{indlimsystemoverKathroughLamdaa})  is
proved. Therefore, the $*$-homomorphism  $\theta$  between the
inductive limits is an isomorphism  of $C^*$-algebras. Thus,
isomorphism (\ref{isoindlimLambdaaKa})  exists.

Furthermore, we obtain  the existence of isomorphism  (\ref{iso1}),
as claimed. Finally, by Lemma~\ref{l4}, there exists a subgroup
$\mathbb{Q}_M$
 in the group  $\mathbb{Q}$  for which we have isomorphism
(\ref{isoCredindlimsystemoverLambda}). Thus, using isomorphism
(\ref{iso1}), we obtain isomorphism  (\ref{mainiso}), as required.
\end{proof}

\section{Inductive systems of Toeplitz algebras over arbitrary partially ordered sets}\label{last}

Throughout this section  a pair   $(\,K, \, \leq \,)$  denotes a
partially ordered set that is not necessarily directed. By analogy
with Definition~\ref{D:indsystemToeplitzalg}, one can define the
inductive system  of Toeplitz algebras over   $(\,K, \, \leq \,)$
defined by  set  (\ref{setN})  satisfying factorization equalities
(\ref{factor}). Below we shall consider such an inductive system
$\mathcal{F}$  and use notation (\ref{mathcalFindsystemoverK}).

Taking the family of all  directed subsets of the set $(\,K, \, \leq
\,)$  and making use of   Zorn's lemma, one can easily prove the
following statement.

\begin{lemma}\label{Krepresentation}
Let \ $(\,K, \, \leq \,)$ \ be a partially ordered set. Then the
following equality holds:
\begin{equation}\label{KbigcupKi}
K=\bigcup\limits_{i\in I}K_i,
\end{equation}
where \ $\left \{ \, K_i \, \mid \, i\in I \, \right \}$ \ is the
family of all maximal  directed subsets of the set  $(\, K, \, \leq
\, )$.
\end{lemma}

Now, for a given inductive system of Toeplitz algebras
(\ref{mathcalFindsystemoverK})  over  the partially ordered set
$(\,K, \, \leq \,)$  defined by set  (\ref{setN})  satisfying
(\ref{factor})  we consider representation  (\ref{KbigcupKi}).
 Then for each  index  $i\in I$  we can construct the inductive system of Toeplitz algebras
\begin{equation}\label{subsystem}
\mathcal{F}_i=(K_i,\{\mathfrak{A}_a\},\{\sigma_{ba}\})
\end{equation}
 over the
 directed set   $\left(\,K_i, \, \leq \,\right)$  defined by the set of natural numbers
$ \left \{\, n_{ba}\in N \, \mid \, a,b \in K_i \, \right \} $ \ and
its inductive limit  $\varinjlim\mathcal{F}_i$.

We consider the direct product of  $C^*$-algebras
$\varinjlim\mathcal{F}_i$. That is, the  $C^*$-algebra
$$
\prod\limits_{i\in I}\varinjlim\mathcal{F}_i:= \left\{ (f_i) \,
\big| \, \|(f_i)\|=\sup_i\|f_i\|< + \infty \right\}
$$
relative to the pointwise operations and the
supremum norm.

To formulate  the result of this section it is convenient for us to give the definition.
\begin{definition}
The inductive system  $\mathcal{F}_i$  defined by (\ref{subsystem})
is called the inductive subsystem of  $\mathcal{F}$ over the set
$K_i$.
\end{definition}

The following statement is an immediate consequence of Theorem~\ref{indlim}.

\begin{theorem}
Let  $K$  be a partially ordered set and let  $\left\{K_i \subset K
\mid i\in I \right\} $  be a family of all maximal directed subsets
in the set  $K$. Let   $\mathcal{F}$  be an inductive system of
Toeplitz algebras  over   $K$  defined  by a  set of natural numbers
 $N$  satisfying factorization equalities. Let   $\mathcal{F}_i$,
where  $i\in I$,   denote the inductive subsystem of  $\mathcal{F}$
  over  the set  $K_i$. Then there exists a family  $ \left
\{\mathbb{Q}_{M_i}\subset \mathbb{Q} \mid i\in I \right \} $
consisting of subgroups in the group of all rational numbers  $
\mathbb{Q}$  and an isomorphism between the direct products of
$C^*$-algebras
$$
\prod\limits_{i\in
I}\varinjlim\mathcal{F}_i\simeq \prod\limits_{i\in
I}C^*_{r}(\mathbb{Q}_{M_i}^+ )
$$
in the category of unital  $C^*$-algebras and their unital
$*$-homomorphisms.
\end{theorem}

\centerline{\textbf{Acknowledgments}}
The author is grateful to S.~A.~Grigoryan and E.~V.~Lipacheva for helpful discussions of the results.

The research was funded by the subsidy allocated to Kazan Federal University
for the state assignment in the sphere of scientific activities, project ¹ 1.13556.2019/13.1.  

\bibliographystyle{amsplain}

\end{document}